\documentclass[12pt]{amsart}
\usepackage[utf8x]{inputenc}
\usepackage{amsmath, amsthm}

 \usepackage[usenames,dvipsnames]{pstricks}
 \usepackage{epsfig}
 \usepackage{pst-grad} 
 \usepackage{pst-plot} 

\usepackage{youngtab}

\newtheorem{theorem}{Theorem}[section]
\newtheorem{lemma}[theorem]{Lemma}
\newtheorem{proposition}[theorem]{Proposition}
\newtheorem{corollary}[theorem]{Corollary}

\theoremstyle{definition}

\newtheorem{example}[theorem]{Example}
\newtheorem{remark}[theorem]{Remark}

\usepackage{amscd,amssymb}

\begin{document}

\title[Cocharacters of block triangular matrices]
{Cocharacters of polynomial identities\\
of block triangular matrices}

\author[Vesselin Drensky, Boyan Kostadinov]
{Vesselin Drensky, Boyan Kostadinov}
\address{Institute of Mathematics and Informatics,
Bulgarian Academy of Sciences, 1113 Sofia, Bulgaria}
\email{drensky@math.bas.bg}
\address{Allianz Bulgaria,
59 Dondukov Boulevard,
1504 Sofia, Bulgaria}
\email{boyan.sv.kostadinov@gmail.com}

\thanks{The updated version of this preprint was partially supported by Grant I02/18
``Computational and Combinatorial Methods
in Algebra and Applications''
of the Bulgarian National Science Fund.}


\maketitle

\begin{abstract}
Let $T(R_{p,q}(K))$ be the T-ideal of the polynomial identities of the
algebra of upper block triangular $(p+2q)\times (p+2q)$ matrices over a field $K$ of
characteristic zero with diagonal consisting of $p$ copies of $1\times 1$ and $q$ copies of $2\times 2$
matrices. We give an algorithm which calculates the
generating function of the cocharacter sequence
$\chi_n(R_{p,q}(K))=\sum_{\lambda\vdash n}m_{\lambda}(R_{p,q}(K))\chi_{\lambda}$
of the T-ideal $T(R_{p,q}(K))$. We have found the
explicit form of the multiplicities $m_{\lambda}(R_{p,q}(K))$ and their asymptotic behaviour
for small values of $p$ and $q$.
\end{abstract}

\section*{Introduction}

We fix a field $K$ of characteristic 0 and consider unital associative algebras over $K$ only.
For a background of PI-algebras and details of the results we refer to the book \cite{D2}.
For a survey on the results related to our paper see \cite{BD}.
Let $R$ be a PI-algebra and let
\[
T(R)\subset K\langle X\rangle=K\langle x_1,x_2,\ldots\rangle
\]
be the T-ideal of its polynomial identities, where $K\langle X\rangle$ is the free associative algebra of countable rank.
One of the most important objects in the quantitative study of the polynomial identities of $R$ is the cocharacter sequence of $R$.
The $n$-th cocharacter $\chi_n(R)$ is equal to the character of the representation of $S_n$ acting on the vector subspace
$P_n\subset K\langle X\rangle$ of the multilinear polynomials of degree $n$ modulo the polynomial identities of $R$:
\[
\chi_n(R)=\sum_{\lambda\vdash n}m_{\lambda}(R)\chi_{\lambda},\quad n=0,1,2,\ldots,
\]
where the summation runs on all partitions $\lambda=(\lambda_1,\ldots,\lambda_n)$ of $n$
and $\chi_{\lambda}$ is the corresponding irreducible character of $S_n$.
The explicit form of the multiplicities $m_{\lambda}(R)$ is known for few algebras only,
among them the Grassmann algebra $E$ (Krakowski and Regev \cite{KR}, Olsson and Regev \cite{OR}),
the $2\times 2$ matrix algebra $M_2(K)$ (Formanek \cite{F1} and Drensky \cite{D1}),
the algebra $U_2(K)$ of the $2\times2$ upper triangular matrices (Mishchenko, Regev and Zaicev \cite{MRZ},
based on the approach of Berele and Regev \cite{BR1}, see also \cite{D2}),
the tensor square $E\otimes E$ of the Grassmann algebra (Popov \cite{P2}, Carini and Di Vincenzo \cite{CDV}),
the algebra $U_2(E)$ of $2\times2$ upper triangular matrices with Grassmann entries (Centrone \cite{Ce}).

One of the possible ways to calculate the cocharacter sequence of $R$ is the following.
The general linear group $GL_d=GL_d(K)$ acts on the $d$-generated free subalgebra
$K\langle X_d\rangle=K\langle x_1,\ldots,x_d\rangle\subset K\langle X\rangle$
and its ideal of the polynomial identities in $d$ variables of $R$ is $GL_d$-invariant.
The relatively free algebra of rank $d$
\[
F_d(R)=K\langle X_d\rangle/(K\langle X_d\rangle\cap T(R))=K\langle x_1,\ldots,x_d\rangle/(K\langle x_1,\ldots,x_d\rangle\cap T(R))
\]
in the variety of algebras ${\rm{var}}(R)$ generated by the algebra $R$ is ${\mathbb Z}^d$-graded with grading defined by
\[
{\rm{deg}}(x_1)=(1,0,\ldots,0), {\rm{deg}}(x_2)=(0,1,\ldots,0),\ldots,{\rm{deg}}(x_d)=(0,0,\ldots,1).
\]
The Hilbert series
\[
H(F_d(R),T_d)=H(F_d(R),t_1,\ldots,t_d)=\sum_{n_i\geq 0}\dim(F_d^{(n_1,\ldots,n_d)}(R))t_1^{n_1}\cdots t_d^{n_d}
\]
of $F_d(R)$, where $F_d^{(n_1,\ldots,n_d)}(R)$ is the homogeneous component of degree $(n_1,\ldots,n_d)$ of $F_d(R)$,
is a symmetric function which plays the role of the character of the corresponding $GL_d$-representation.
The Schur functions $S_{\lambda}(T_d)=S_{\lambda}(t_1,\ldots,t_d)$ are the characters of the irreducible $GL_d$-submodules
$W_d(\lambda)$ of $F_d(R)$ and
\[
H(F_d(R),T_d)=\sum_{\lambda}m_{\lambda}(R)S_{\lambda}(T_d),\quad \lambda=(\lambda_1,\ldots,\lambda_d).
\]
By a result of Berele \cite{B3} and Drensky \cite{D1, D4}, the multiplicities $m_{\lambda}(R)$ are the same as
in the cocharacter sequence $\chi_n(R)$, $n=0,1,2,\ldots$. Hence, in principle, if we know the Hilbert series $H(F_d(R),T_d)$,
we can find the multiplicities $m_{\lambda}(R)$ in $\chi_n(R)$ for the partitions $\lambda$ in not more than $d$ parts.
When $R$ is a finite dimensional algebra, the multiplicities $m_{\lambda}(R)$ are equal to zero for partitions
$\lambda=(\lambda_1,\ldots,\lambda_d)$, $\lambda_d\neq0$, for $d>{\rm{dim}}(R)$, see Regev \cite{R2}.
Hence, all $m_{\lambda}(R)$ can be recovered from $H(F_d(R),T_d)$ for $d$ sufficiently large.
Following the idea of Drensky and Genov \cite{DG1} we consider the multiplicity series of $R$
\[
M(R;T_d)=M(R;t_1,\ldots,t_d)=\sum_{\lambda}m_{\lambda}(R)T_d^{\lambda}=\sum_{\lambda}m_{\lambda}(R)t_1^{\lambda_1}\cdots t_d^{\lambda_d},
\]
i.e., the generating function of the cocharacter sequence of $R$ for the partitions in $\leq d$ parts.
Then, if we know the Hilbert series $H(F_d(R),T_d)$, the problem is to compute the multiplicity series $M(R;T_d)$
and to find its coefficients. This problem was solved in \cite {DG2} for rational symmetric functions of special kind and in two variables.
Berele \cite{B1}, see also Berele and Regev \cite{BR2} and Berele \cite{B2}, suggested an approach involving the
so called nice rational functions.
But the approach of \cite{B1, BR2, B2} does not give explicit algorithms to find the multiplicities of the irreducible characters.
One can apply classical algorithms to find the multiplicity series of a given symmetric function based on
a method of Elliott \cite{E}, improved by MacMahon \cite{Mc} in his ``$\Omega$-Calculus'' or Partition Analysis,
with further improvements and computer realizations, see Andrews, Paule and Riese \cite{APR} and Xin \cite{X}.

Formanek \cite{F2} expressed the Hilbert series of the product of two T-ideals in terms
of Hilbert series of the factors. Berele and Regev \cite{BR1} translated this result in the language of cocharacters.
Hence, in principle, if we know the Hilbert series of $F_d(R_1)$ and $F_d(R_2)$, we can find the multiplicities
of the cocharacter series of $T(R)=T(R_1)T(R_2)$. By a theorem of Maltsev \cite{Ma}
the T-ideal of the algebra $U_k(K)$ of $k\times k$ upper triangular matrices with entries from the field $K$
is
\[
T(U_k(K))=T(K)^k,
\]
where $T(K)$ is the commutator ideal of the free algebra. Boumova and Drensky \cite{BD}
developed methods to express the multiplicity series of the cocharacters of $TU_k(K)$
and evaluated them for small $k$.

In the present paper we consider the algebra $R_{p,q}(K)$ of upper block triangular matrices,
with $p$ blocks of size $1\times 1$ and $q$ blocks of size $2\times 2$ on the main diagonal.
We present a formula for the Hilbert series $H(F_d(R_{p,q}(K)),T_d)$ of the relatively free algebra $F_d(R_{p,q}(K))$.
For small values of $p$ and $q$, we calculate the multiplicity series and the explicit form of the multiplicities.
We have handled the following cases:

\begin{itemize}
\item For $p=1$, $q=1$ the multiplicity series in any number of variables and the multiplicities for arbitrary partitions;

\item For $p=0$, $q=2$ for three variables and partitions in three parts;

\item For $p=0$, $q=3$ and $p=0$, $q=4$ for two variables and partitions in two parts.
\end{itemize}

In all these cases we determine the asymptotic behaviour of the multiplicities $m_{\lambda}(R_{p,q})$.
A key role in our considerations is the calculation of the multiplicity series of the powers of the symmetric function
\[
f=\sum_{n\geq 0}S_{(n,n)}(T_d)
\]
which has the following simple multiplicity series
\[
M(f;T_d)=\frac{1}{1-t_1t_2}.
\]

The results of this paper are part of the Master Thesis \cite{K}
of the second-named author defended in 2011 at the Faculty of Mathematics and Informatics
of the University of Sofia and supervised by the first-named author.

\section{Hilbert series}
To simplify the notation, sometimes we shall omit the variables $T_d$ in the formal power series.
For example, we shall write $H(K \langle X_d \rangle \cap T(R))$, $H(F_d(R))$, etc. instead of
$H(K \langle X_d \rangle \cap T(R), T_d)$, $H(F_d(R),T_d)$.

We use the following result of Formanek \cite{F2}, see Halpin \cite{Ha} for the proof.

\begin{theorem}\label{Theorem of Formanek}
Let $R_1, R_2$ and $R$ be $PI$-algebras such that $T(R)=T(R_1)T(R_2)$. Then the Hilbert series
of $T(R_1)$, $T(R_2)$ and $T(R)$ are related by
\[
H(K \langle X_d \rangle)H(K \langle X_d \rangle \cap T(R))=\frac{1}{1-(t_1+\cdots+t_d)}H(K \langle X_d \rangle \cap T(R))
\]
\[
=H(K \langle X_d \rangle \cap T(R_1))H(K \langle X_d \rangle \cap T(R_2)).
\]
\end{theorem}

\begin{corollary}\label{Corollary of theorem by Formanek}
{\rm (i)} In the notation of the previous theorem
\[
H(F_d(R))=H(F_d(R_1))+H(F_d(R_2))+(t_1+\cdots+t_d-1)H(F_d(R_1))H(F_d(R_2));
\]
{\rm (ii) (Berele and Regev \cite{BR1})}
If $\chi_n(R_1)$ and $\chi_n(R_2)$, $n=0,1,2,\ldots$, are, respectively, the cocharacter sequences of the algebras $R_1$ and $R_2$,
then the cocharacter sequence of the algebra $R$ with T-ideal $T(R)=T(R_1)T(R_2)$ is
\[
\chi_n(R)=\chi_n(R_1)+\chi_n(R_2)+\chi_{(1)}\widehat\otimes\sum_{j=0}^{n-1}\chi_j(R_1)\widehat\otimes\chi_{n-j-1}(R_2)
\]
\[
-\sum_{j=0}^n\chi_j(R_1)\widehat\otimes\chi_{n-j}(R_2),
\]
where $\widehat\otimes$ denotes the ``outer'' tensor product of characters.
\end{corollary}

The first part of the corollary follows immediately from Theorem \ref{Theorem of Formanek}
taking into account that
\[
H(K\langle X_d\rangle)=H(F_d(R))+H(K\langle X_d\rangle\cap T(R))
\]
and similarly for $R_1$ and $R_2$. In the second part of the corollary,
the outer product of irreducible characters corresponds to the Littlewood-Richardson rule for products of Schur functions. If
\[
S_{\lambda}(T_d)S_{\mu}(T_d)=\sum_{\nu\vdash |\lambda|+|\mu|}c_{\lambda\mu}^{\nu}S_{\nu}(T_d),
\]
\[
\lambda=(\lambda_1,\ldots,\lambda_p),\quad \mu=(\mu_1,\ldots,\mu_q),\quad d\geq p+q,
\]
then
\[
\chi_{\lambda}\widehat\otimes\chi_{\mu}=\sum_{\nu\vdash |\lambda|+|\mu|}c_{\lambda\mu}^{\nu}\chi_{\nu}.
\]

We obtain the following corollary.

\begin{corollary}\label{Hilbert series of product of several ideals}
Let $R_1,\ldots,R_n$ and $R$ be PI-algebras such that $T(R)=T(R_1)\cdots T(R_n)$.

{\rm (i)} The Hilbert series of the relatively free algebra $F_d(R)$ is given by the formula
\[
H(F_d(R))=\sum_{k=1}^{n}(t_1+\cdots+t_d-1)^{k-1}
\sum_{1\leq i_1<\cdots<i_k\leq n}H(F_d(R_{i_1}))\cdots H(F_d(R_{i_k})).
\]

{\rm (ii)} For a permutation $\sigma\in S_n$ let $R_{\sigma}$ be an algebra with T-ideal
$T(R_{\sigma})=T(R_{\sigma(1)})\cdots T(R_{\sigma(n)})$. Then
$H(F_d(R_{\sigma}))=H(F_d(R))$.

{\rm (iii)}
If $R_1=\cdots=R_n$, then the above formula has the form
\begin{equation}\label{n equal}
H(F_d(R))=\sum_{k=1}^{n}\binom{n}{k}(t_1+\cdots+t_d-1)^{k-1}H(F_d(R_1))^k.
\end{equation}
\end{corollary}

\begin{proof}
The equation in (i) follows directly applying several times Corollary \ref{Corollary of theorem by Formanek} (i).
Part (ii) is an obvious consequence of (i).
To derive (iii), we note that the number of summands $H(F_d(R_{i_1}))\cdots H(F_d(R_{i_k}))$
in each of the sums in the first equation is $\binom{n}{k}$. Hence, for $R_1=R_2=\cdots=R_n$ we obtain (iii).
\end{proof}

Let $R_1$ and $R_2$ be PI-algebras and let $M$ be an $R_1$-$R_2$-bimodule. Then the polynomial identities of the algebra
\[
R=\left(\begin{matrix}
R_1&M\\
0&R_2\\
\end{matrix}\right)
\]
satisfy $T(R)\supseteq T(R_1)T(R_2)$. Lewin \cite{L} found conditions which guarantee that $T(R)=T(R_1)T(R_2)$.
A nontrivial consequence of the result of Lewin is the following result of Giambruno and Zaicev \cite{GZ}
which is crucial for our subsequent considerations. Let $d_1,\ldots,d_m$ be positive integers and let $U(d_1,\ldots,d_m)$
be the algebra of upper block triangular matrices of the form
\[
\begin{pmatrix}
M_{d_1}(K) &\ast & \dots &\ast & \ast \\
0 & M_{d_2}(K) &\dots &\ast &\ast \\
\vdots & \vdots& \ddots &\vdots & \vdots \\
0&0&\dots&M_{d_{m-1}}(K)&\ast\\
0 &0& \dots & 0 & M_{d_m}(K)\\
\end{pmatrix},
\]

\bigskip
\noindent where $M_{d_i}(K)$ is the algebra of $d_i\times d_i$-matrices over $K$.

\begin{proposition}\label{Giambruno and Zaicev}{\rm (Giambruno and Zaicev \cite{GZ})}
The T-ideal of the polynomial identities of the algebra $U(d_1,\ldots,d_m)$ of upper block triangular matrices
equals the product of T-ideals of all matrices $M_{d_i}(K)$, $i=1,\ldots,m$:
\[
T(U(d_1,\ldots,d_m))=T(M_{d_1}(K))\cdots T(M_{d_m}(K)).
\]
\end{proposition}

Now let us consider the algebra $R_{p,q}=R_{p,q}(K)=U(d_1,\ldots,d_{p+q})$ with diagonal consisting of $p$ copies
of $1\times 1$ and $q$ copies of $2\times 2$ blocks only.

\begin{proposition}\label{Algebra Rpq}
The Hilbert series $H(F_d(R_{p,q}))$ of the algebra $F_d(R_{p,q})$ is
\[
H(F_d(R_{p,q}))=\sum_{i=1}^{p} \binom{p}{i}(t_1+\cdots+t_d-1)^{i-1}H(K[X_d])^i
\]
\[
+\sum_{j=1}^{q}\binom{q}{j}(t_1+\cdots+t_d-1)^{j-1}H(F_d(M_2(K)))^j
\]
\[
+\sum_{i=1}^{p}\sum_{j=1}^{q}\binom{p}{i}\binom{q}{j}(t_1+\cdots+t_d-1)^{i+j-1}H(K[X_d])^iH(F_d(M_2(K)))^j.
\]
\end{proposition}

\begin{proof}
By Proposition \ref{Giambruno and Zaicev}
we conclude that
\[
T(R_{p,q})=T(U(d_1,\ldots,d_{p+q}))=T(M_{d_1}(K))\cdots T(M_{d_{p+q}}(K)),
\]
where $p$ of the integers $d_i$ are equal to 1 and the other $q$ are equal to 2.
Corollary \ref{Hilbert series of product of several ideals} (ii) implies that
\[
H(F_d(R_{p,q}))=H(F_d(M_1(K)))^pH(F_d(M_2(K)))^q.
\]
Hence the Hilbert series of $F_d(R_{p,q})$ is the same as the Hilbert series of
$F_d(R'_{p,q})$, where
\[
R'_{p,q}=U(\underbrace{1,\ldots,1}_{p\text{ times}},\underbrace{2\ldots,2}_{q\text{ times}})
\]
and $T(R'_{p,q})=T(R_{(1)})T(R_{(2)})$ with $T(R_{(1)})=T(M_1(K))^p$, $T(R_{(2)})=T(M_2(K))^q$.
Then by Corollary \ref{Corollary of theorem by Formanek} (i) we have
\[
H(F_d(R_{p,q}))=H(F_d(R_{(1)}))+H(F_d(R_{(2)})
\]
\[
+(t_1+\cdots+t_d-1)H(F_d(R_{(1)}))H(F_d(R_{(2)})).
\]
To complete the proof we use Corollary \ref{Hilbert series of product of several ideals} (iii)
taking into account that $F_d(M_1(K))=K[X_d]$, the polynomial algebra in $d$ commuting variables.
\end{proof}

\section{Symmetric functions}

For a background on symmetric functions see the book by Macdonald \cite{McD}.
Recall that one of the ways to define the Schur function $S_{\lambda}=S_{\lambda}(T_d)$ is as a fraction of determinants of
Vandermonde type
\[
S_{\lambda}(T_d)=\frac{V(\lambda+\delta)}{V(\delta)},
\]
where $\delta=(d-1,\ldots,2,1)$ and for $\mu=(\mu_1,\ldots,\mu_d)$
\[
V(\mu,T_d)=\left\vert\begin{matrix}
t_1^{\mu_1}&t_2^{\mu_1}&\cdots&t_d^{\mu_1}\\
&&&\\
t_1^{\mu_2}&t_2^{\mu_2}&\cdots&t_d^{\mu_2}\\
&&&\\
\vdots&\vdots&\ddots&\vdots\\
&&&\\
t_1^{\mu_d}&t_2^{\mu_d}&\cdots&t_d^{\mu_d}\\
\end{matrix}\right\vert.
\]
It is well known that Schur functions form a basis of the algebra of symmetric polynomials
${\mathbb C}[T_d]^{S_d}$. Hence, if $f(T_d)={\mathbb C}[[T_d]]^{S_d}$ is a symmetric function
in $d$ variables which is a formal power series, then
\[
f(T_d)=\sum_{\lambda}m_{\lambda}S_{\lambda}(T_d),\quad m_{\lambda}\in{\mathbb C}.
\]
We associate with $f(T_d)$ its multiplicity series
\[
M(f;T_d)=\sum_{\lambda}m_{\lambda}T_d^{\lambda}
=\sum_{\lambda}m_{\lambda}t_1^{\lambda_1}\cdots t_d^{\lambda_d}\in {\mathbb C}[[T_d]].
\]
We consider also the subalgebra ${\mathbb C}[[V_d]]\subset {\mathbb C}[[T_d]]$
of the formal power series in the new set of variables $V_d=\{v_1,\ldots,v_d\}$, where
\[
v_1=t_1,v_2=t_1t_2,\ldots,v_d=t_1\cdots t_d.
\]
Then the multiplicity series $M(f;T_d)$ can be written as
\[
M'(f;V_d)=\sum_{\lambda}m_{\lambda}v_1^{\lambda_1-\lambda_2}\cdots
v_{d-1}^{\lambda_{d-1}-\lambda_d}v_d^{\lambda_d}\in {\mathbb C}[[V_d]]
\]
and the mapping
$M':{\mathbb C}[[T_d]]^{S_d}\to {\mathbb C}[[V_d]]$ defined by
$M':f(T_d)\to M'(f;V_d)$ is a bijection.
For a PI-algebra $R$ we define the multiplicity series of $R$
\[
M(R;T_d)=M(R;t_1,\ldots,t_d)= \sum_{\lambda}m_{\lambda}(R)T_d^{\lambda}
=\sum_{\lambda}m_{\lambda}(R)t_1^{\lambda_1}\cdots t_d^{\lambda_d}.
\]
Similarly we define the series $M'(R;V_d)$.

\begin{lemma}\label{relation between H and M}
{\rm (Berele \cite{B1})}
The functions $f(T_d)\in {\mathbb C}[[T_d]]^{S_d}$ and $M(f;T_d)$ are related in the following way:

{\rm (i)} If
\[
f(T_d)\prod_{i<j}(t_i-t_j)=\sum_{p_i\geq 0}b(p_1,\ldots,p_d)t_1^{p_1}\cdots t_d^{p_d},
\quad b(p_1,\ldots,p_d)\in {\mathbb C},
\]
then
\[
M(f;T_d)=\frac{1}{t_1^{d-1}\cdots t_{d-2}^2t_{d-1}}
\sum_{p_i>p_{i+1}}b(p_1,\ldots,p_d)t_1^{p_1}\cdots t_d^{p_d},
\]
where the summation is on all $p=(p_1,\ldots,p_d)$ such that
$p_1>p_2>\cdots>p_d$.

{\rm (ii)} The formal power series
\[
h(T_d)=\sum h(q_1,\ldots,q_d)t_1^{q_1}\cdots t_d^{q_d},\quad
q_1\geq\cdots\geq q_d,
\]
is equal to the multiplicity series $M(f;T_d)$ of $f(T_d)$ if and only if
\[
f(T_d)\prod_{i<j}(t_i-t_j)=\sum_{\sigma\in S_d}\text{\rm sign}(\sigma)
t_{\sigma(1)}^{d-1}t_{\sigma(2)}^{d-2}\cdots t_{\sigma(d-1)}
h(t_{\sigma(1)},\ldots,t_{\sigma(d)}).
\]
\end{lemma}

Usually, it is difficult to find $M(f;T_d)$ if we know $f(T_d)$. But
using Lemma \ref{relation between H and M} (ii) it is very easy to check whether a formal power series
is equal to the multiplicity series.
This argument can be used to verify the computational results on multiplicities.

Below we shall summarize some methods to compute the multiplicity series of symmetric functions of special type.

If two symmetric functions $f(T_d)$ and $g(T_d)$ are related by
\[
f(T_d)=g(T_d)\prod_{i=1}^d\frac{1}{1-t_i}=g(T_d)\sum_{n\geq 0}S_{(n)}(T_d),
\]
then $f(T_d)$ is Young derived from $g(T_d)$ because the product of Schur functions $S_{\mu}(T_d)S_{(n)}(T_d)$  can be
decomposed using the Young rule. This is a translation of the notion of Young derived sequences of $S_n$-charactres
introduced by Regev \cite{R3}.
Let $Y$ be the linear operator in ${\mathbb C}[[V_d]]\subset {\mathbb C}[[T_d]]$ which sends the multiplicity series of
the symmetric function $g(T_d)$ to the multiplicity series of its Young derived
$f(T_d)$:
\[
Y(M(g);T_d)=M(f;T_d)=M\left(g(T_d)\prod_{i=1}^d\frac{1}{1-t_i};T_d\right).
\]

\begin{example}\label{multiplicity series of 2x2 matrices}
The multiplicities of the cocharacter sequence of the $2\times 2$ matrix algebra $M_2(K)$
and the Hilbert series of the relatively free algebra $F_d(M_2(K))$ were computed by Procesi \cite{Pr},
Formanek \cite{F1} and Drensky \cite{D1}, see also \cite{D2}. Following \cite{D1, D2} we have
\[
H(F_d(M_2(K)))=\prod_{i=1}^{d}\frac{1}{1-t_i}\left(\sum_{(\lambda_1,\lambda_2,\lambda_3)}S_{(\lambda_1,\lambda_2,\lambda_3)}
-S_{(1,1,1)}-\sum_{n\geq 1}S_{(n)}\right).
\]
By the Young rule we have
\[
\sum_{(\lambda_1,\lambda_2,\lambda_3)}S_{(\lambda_1,\lambda_2,\lambda_3)}=\prod_{i=1}^{d}\frac{1}{1-t_i}\sum_{n\geq 0}S_{(n,n)}.
\]
Also,
\[
\sum_{n\geq 1}S_{(n)}=\sum_{n\geq 0}S_{(n)}-1=\prod_{i=1}^{d}\frac{1}{1-t_i}-1
\]
and hence
\[
H(F_d(M_2(K)))=\prod_{i=1}^{d}\frac{1}{1-t_i}
\left(\prod_{i=1}^{d}\frac{1}{1-t_i}\sum_{n\geq 0}S_{(n,n)}
+1-S_{(1,1,1)}-\prod_{i=1}^{d}\frac{1}{1-t_i}\right)
\]
\[
=\prod_{i=1}^{d}\frac{1}{(1-t_i)^2}\left(\sum_{n\geq 0}S_{(n,n)}-1\right)
+\prod_{i=1}^{d}\frac{1}{1-t_i}(1-S_{(1,1,1)}).
\]
Since
\[
M\left(\sum_{n\geq 0}S_{(n,n)}-1\right)=\sum_{n\geq 0}(t_1t_2)^n-1=\frac{1}{1-t_1t_2}-1
\]
we express the multiplicity series of $M_2(K)$ as
\[
M(M_2(K);T_d)=Y^2\left(\frac{1}{1-t_1t_2}-1\right)+Y(1-t_1t_2t_3).
\]
\end{example}

\begin{example}\label{multiplicities of upper triangular matrices}
By Boumova and Drensky \cite{BD}
the multiplicity series of the algebra $U_k(K)$ of $k\times k$ upper triangular matrices with entries from the base field is
\[
M(U_k;T_d)=\sum_{j=1}^k\sum_{q=0}^{j-1}\sum_{\lambda\vdash q}(-1)^{j-q-1}\binom{k}{j}\binom{j-1}{q}d_{\lambda}Y^j(T_d^{\lambda}),
\]
where $d_{\lambda}$ is the degree of the irreducible $S_n$-character $\chi_{\lambda}$.
\end{example}

The combination of Proposition \ref{Algebra Rpq} and Example \ref{multiplicity series of 2x2 matrices}
immediately gives the explicit form of the Hilbert series of the algebra $R_{p,q}$.

\begin{theorem}\label{Hilbert series of Rpq}
The Hilbert series $H(F_d(R_{p,q}))$ of the algebra $F_d(R_{p,q})$ is
\[
H(F_d(R_{p,q}(K)),T_d)=\sum_{i=1}^{p} \binom{p}{i}(t_1+\cdots+t_d-1)^{i-1}\left( \prod_{s=1}^{d}\frac{1}{1-t_s}\right)^i
\]
\[
+\sum_{j=1}^{q}\binom{q}{j}(t_1+\cdots+t_d-1)^{j-1}\left( \prod_{s=1}^{d}\frac{1}{1-t_s}\right)^j \times
\]
\[
\times \left( \sum_{m\geq 0}S_{(m,m)}\prod_{s=1}^{d}\frac{1}{1-t_s}-S_{(1^3)}-\prod_{s=1}^{d}\frac{1}{1-t_s}+1 \right)^j
\]
\[
+\sum_{i=1}^{p}\sum_{j=1}^{q}\binom{p}{i}\binom{q}{j}(t_1+\cdots+t_d-1)^{i+j-1}\left( \prod_{s=1}^{d}\frac{1}{1-t_s}\right)^{i+j}\times
\]
\[
\times\left( \sum_{m\geq 0}S_{(m,m)}\prod_{s=1}^{d}\frac{1}{1-t_s}-S_{(1^3)}-\prod_{s=1}^{d}\frac{1}{1-t_s}+1 \right)^j.
\]
\end{theorem}

\begin{remark}\label{computing of multiplicities of Rpq}
In the next section we shall use Theorem \ref{Hilbert series of Rpq} for concrete computations with the multiplicities
of the algebra $R_{p,q}$ for small $p$ and $q$.
In the spirit of Example \ref{multiplicities of upper triangular matrices}, we may express the multiplicity series of $R_{p,q}$
applying on the multiplicity series of the symmetric functions
\[
\left(\sum_{m\geq 0}S_{(m,m)}(T_d)\right)^j,\quad j=0,1,\ldots,q,
\]
powers of the operator $Y$ and the operators $Y_{(1)}$ and $Y_{(1^3)}$
in ${\mathbb C}[[V_d]]\subset {\mathbb C}[[T_d]]$ defined by
\[
Y_{(1^k)}(M(g);T_d)=M(g(T_d)S_{(1^k)}(T_d)),\quad g\in {\mathbb C}[[T_d]]^{S_d}.
\]
\end{remark}

The following proposition shows the action of the operator $Y$ in the language of multiplicity series.

\begin{proposition}\label{multiplicity series of Young derives}
{\rm (Drensky and Genov \cite{DG1})}
Let $f(T_d)$ be the Young derived of the symmetric function $g(T_d)$.
Then
\[
Y(M(g);T_d)=M(f;T_d)=M\left(g(T_d)\prod_{i=1}^d\frac{1}{1-t_i};T_d\right)
\]
\[
=\prod_{i=1}^d\frac{1}{1-t_i}\sum(-t_2)^{\varepsilon_2}\ldots(-t_d)^{\varepsilon_d}
M(g;t_1t_2^{\varepsilon_2},t_2^{1-\varepsilon_2}t_3^{\varepsilon_3}\ldots
t_{d-1}^{1-\varepsilon_{d-1}}t_d^{\varepsilon_d},t_d^{1-\varepsilon_d}),
\]
where the summation runs on all $\varepsilon_2,\ldots,\varepsilon_d=0,1$.
\end{proposition}

The decomposition of the product $S_{\mu}S_{(1)}$ is given by the Branching theorem which a
special case of the Young rule. It states that
\[
S_{\mu}(T_d)S_{(1)}(T_d)=\sum S_{\lambda}(T_d),
\]
where the summation is on all partitions $\lambda$ of
$\vert\mu\vert+1$ such that $\lambda_1\geq\mu_1\geq\lambda_2\geq\mu_2\geq\cdots\geq\lambda_d\geq\mu_d$.
In the language of Young diagrams, this means that the diagram of $\lambda$ is obtained by adding a box
to the diagram of $\mu$.
We shall state it in terms of the multiplicity series $M'(f;V_d)$. The result can be easily restated in terms of
the operator $Y_{(1)}$ defined in Remark \ref{computing of multiplicities of Rpq}.

\begin{lemma}\label{applying the branching theorem}
\[
M'(f(T_d)S_{(1)}(T_d);V_d)=v_1M'(f(T_d);V_d)
\]
\[
+\sum_{i=1}^{d-1}\frac{v_{i+1}}{v_i}\left(M'(f(T_d);v_1,\ldots,v_d)-M'(f(T_d);v_1,\ldots,v_{i-1},0,v_{i+1},\ldots,v_d)\right).
\]
\end{lemma}

\begin{proof}
It is sufficient to prove the lemma for $f(T_d)=S_{\mu}(T_d)$. Rewriting $\mu$ as
\[
\mu=(p_1+\cdots+p_d,p_2+\cdots+p_d,\ldots,p_d),
\]
where $p_i=\mu_i-\mu_{i+1}$, $i=1,\ldots,d-1$, $p_d=\mu_d$,
we obtain
\[
M'(S_{\mu}(T_d),V_d)=v_1^{p_1}\cdots v_d^{p_d}.
\]
By the Branching theorem $S_{\mu}(T_d)S_{(1)}(T_d)$ is a sum of $S_{\lambda}(T_d)$, where
\[
\lambda=(\mu_1,\ldots,\mu_i,\mu_{i+1}+1,\mu_{i+2},\ldots,\mu_d)
\]
and the sum runs on those $i$ with $\mu_i>\mu_{i+1}$, adding the case $i=0$. (We may add a box to
the first row of the diagram of $\mu$ and to the $(i+1)$-st row, if the boxes on the $i$-th row of
the diagram of $\mu$ are more than the boxes of the $(i+1)$-st row.)
In the language of multiplicity series this means that
$M'(S_{\mu}(T_d)S_{(1)}(T_d),V_d)$ is a sum of $v_1^{p_1+1}v_2^{p_2}\cdots v_d^{p_d}$ (adding a box to the first row of
the diagram of $\mu$) and of those
\[
v_1^{p_1}\cdots v_{i-1}^{p_{i-1}}v_i^{p_i-1}v_{i+1}^{p_{i+1}+1}v_{i+2}^{p_{i+2}}\cdots v_d^{p_d}
=\frac{v_{i+1}}{v_i}v_1^{p_1}\cdots v_d^{p_d}
\]
with $p_i>0$ (to add a box to the $(i+1)$-st row we need $\mu_i>\mu_{i+1}$). This completes the proof because
\[
M'(S_{\mu};V_d)-M'(S_{\mu};v_1,\ldots,v_{i-1},0,v_{i+1},\ldots,v_d)=\begin{cases}
M'(S_{\mu};V_d),\text{ if }p_i>0,\\
0,\text{ if } p_i=0.\\
\end{cases}
\]
\end{proof}

The product $S_{\mu}(T_d)S_{(1^3)}(T_d)$ can be decomposed using the Young rule. The result is the sum of $S_{\lambda}(T_d)$
such that $\vert\lambda\vert=\vert\mu\vert+3$ and the diagram of $\lambda$ is obtained from the diagram of $\mu$ by adding three
boxes to pairwise different rows.
We can establish an analogue of Lemma \ref{applying the branching theorem} which relates the multiplicity series
of the symmetric functions $g(T_d)$ and $g(T_d)S_{(1^3)}(T_d)$ (which gives a formula for the action of the operator $Y_{(1^3)}$)
but the expression is more complicated. For our computations in the next section we do not need the general form
of the action of $Y_{(1^3)}$. It is defined for $d\geq 3$ and we apply it for $d=3$ only. Since
\[
S_{(\mu_1,\mu_2,\mu_3)}(T_3)S_{(1^3)}(T_3)=S_{(\mu_1+1,\mu_2+1,\mu_3+1)}(T_3)
\]
we obtain immediately that
\[
M'(g(T_3)S_{(1^3)}(T_3);V_3)=v_3M'(g(T_3);V_3).
\]

Finally, to compute $M(R_{p,q})$ we need the multiplicity series of the symmetric function
$\left(\sum_{n\geq 0}S_{(n,n)}\right)^j$. We give the results for small $j$.

\begin{lemma}\label{M of powers of f}
Let
\[
f(T_d)=\sum_{n\geq 0}S_{(n,n)}(T_d).
\]
Then
\[
M'(f;V_d)=\frac{1}{1-v_2},
\]
\[
M'(f^2)=\frac{1}{(1-v_1v_3)(1-v_2)^2(1-v_4)}.
\]
For $d=4$
\[
M'(f^3;V_4)=\frac{(1-v_1v_2v_3)(1-v_1^2v_3^2v_4)}{(1+v_3)(1-v_1^2v_4)(1-v_1v_3)^3(1-v_3)(1-v_2)^3(1-v_4)^3}.
\]
\end{lemma}

\begin{proof}
The expression for $M'(f;V_d)$ follows immediately by the definition of the operator $M'$:
\[
M'(f;V_d)=\sum_{n\geq 0}v_2^n=\frac{1}{1-v_2}.
\]
By the Littlewood-Richardson rule, if $\lambda$ and $\mu$ are partitions in $a$ and $b$ parts, respectively, then
the product of the Schur functions $S_{\lambda}$ and $S_{\mu}$
is a linear combination of Schur functions $S_{\nu}$ where $\nu$ is a partition in $\leq a+b$ parts.
Hence, in order to compute $M'(f^2)$ it is sufficient to assume that $d=4$.
Using the expression of Schur functions as quotients of Vandermonde type determinants
we obtain
\[
f(T_4)=\frac{1-t_1t_2t_3t_4}{(1-t_1t_2)(1-t_1t_3)(1-t_1t_4)(1-t_2t_3)(1-t_2t_4)(1-t_3t_4)}.
\]
This expression can be obtained also from the well known equality
\[
\prod_{i<j}\frac{1}{1-t_it_j}=\sum S_{\nu},
\]
where $\nu=(\nu_1^2,\ldots,\nu_k^2)$ (each row of the diagram of $\nu$ appears even number of times).
Then
\[
\prod_{1\leq i<j\leq 4}\frac{1}{1-t_it_j}=\sum_{\nu_1\geq\nu_2\geq 0}S_{(\nu_1,\nu_1,\nu_2,\nu_2)}(T_4)
\]
and the product
\[
t_1t_2t_3t_4\prod_{1\leq i<j\leq 4}\frac{1}{1-t_it_j}=\sum_{\nu_1\geq\nu_2\geq 0}S_{(\nu_1+1,\nu_1+1,\nu_2+1,\nu_2+1)}(T_4)
\]
contains the summands with $\nu_2>0$. Now the equalities in the statement of the lemma can be verified applying
Lemma \ref{relation between H and M} (ii).

We shall calculate once again $M(f^2)$ using the Littlewood-Richardson rule.
Recall that a lattice permutation is a sequence of integers such that in every initial part of the sequence any number $i$ occurs
at least as often as the number $i+1$. We fix a partition $\lambda$ of $2k$. We consider all possible pairs $(m,n)$
such that the diagram of $\lambda$ contains the diagram of the partition $(m,m)$.
We have to count semistandard $\lambda-(m,m)$ tableaux filled in with $n$ numbers 1 and $n$ numbers 2
in such a way that reading its entries from right to left and from top to bottom we
have a lattice permutation.
\bigskip
\begin{center}
{\normalsize \Yvcentermath1
$\displaystyle \sum_{m,n\geq 0}\young(XX\cdots XX,XX\cdots XX)\times \young(11\cdots11,22\cdots 22)$}
\end{center}
\bigskip
\begin{center}
{\normalsize \Yvcentermath1
$\displaystyle =\sum\young(X\ldots XX\ldots XX1\ldots11\ldots1,X\ldots XX\ldots XX2\ldots2,1\ldots12\ldots2,2\ldots2)$}
\end{center}
\bigskip
\noindent Combining the semistandard and lattice permutation properties of the tableau, we see that the only possibilities for the first two rows
are as given in the figure. Since for the first two rows we have used $\lambda_1-m$ numbers 1 and $\lambda_2-m$ numbers 2,
we have left $n-(\lambda_1-m)$ numbers 1 and $n-(\lambda_2-m)$ numbers 2 to fill in the third and the fourth rows.
It is easy to see that the only possibility is as given in the figure which implies the condition $\lambda_1-\lambda_2=\lambda_3-\lambda_4$.
The number of all possibilities is $\lambda_2-\lambda_3+1$ (we can fill in with $2$ part of the boxes of the second row
under which there are no boxes of the third row). Therefore
\[
M'(f^2;V_4)=\sum v_1^{\lambda_1-\lambda_2}(\lambda_2-\lambda_3+1)v_2^{\lambda_2-\lambda_3}v_3^{\lambda_3-\lambda_4}v_4^{\lambda_4},
\]
where the summation runs over all $\lambda=(\lambda_1,\lambda_2,\lambda_3,\lambda_4)$ with $\lambda_1-\lambda_2=\lambda_3-\lambda_4$. Hence,
\[
M'(f^2;V_4)=\sum_{\lambda_1\geq \lambda_2\geq 0} (v_1v_3)^{\lambda_1-\lambda_2}
\sum_{\lambda_2\geq\lambda_3\geq 0}(\lambda_2-\lambda_3+1)v_2^{\lambda_2-\lambda_3}
\sum_{\lambda_4\geq 0}v_4^{\lambda_4}
\]
\[
=\frac{1}{(1-v_1v_3)(1-v_4)}\sum_{k\geq 0}(k+1)v_2^k
=\frac{1}{(1-v_1v_3)(1-v_2)^2(1-v_4)}.
\]
\end{proof}

\begin{remark}\label{how to find the multiplicities}
The formula for $M'(f^3;V_4)$ in Lemma \ref{M of powers of f} is given by an ``ansatz'', and then proved by direct verification.
To find it, we have used a method which is a version of the Xin algorithm \cite{X} for the $\Omega$-Calculus of MacMahon \cite{Mc}
adapted to work with rational symmetric functions. This method has been applied to solve problems in the theory of PI-algebras and
classical invariant theory, see \cite{BBDGK}.
\end{remark}

\section{Results}

Using Theorem \ref{Hilbert series of Rpq}, we have calculated the multiplicity series
of the algebra $R_{p,q}$ for small values of $p$ and $q$.
In most of the cases we give the results for $d=2$ or $d=3$ only. Then we have expanded $M'(R_{p,q};V_d)$
to obtain the explicit form of the multiplicities $m_{\lambda}(R_{p,q})$ and their asymptotics.
Since the formulas for the multiplicities are quite complicated, in the paper we present the asymptotics only.
The multiplicity series have been computed as suggested in Remark \ref{computing of multiplicities of Rpq},
using standard functions of {\it Maple}.
But, once we have their explicit form, we can easily prove that this is the true expression for $M'(R_{p,q};V_d)$
applying Lemma \ref{relation between H and M} (ii). We state the results of our computations as a series of theorems.
In view of the above comments, we omit the proofs. In order to simplify the notation, we make the convention that
$\lambda=(\lambda_1,\ldots,\lambda_d)$ is a partition of
$n$ and define $n_i=\lambda_i-\lambda_{i+1}$, for $i=1,\ldots,d-1$ and $n_d=\lambda_d$.
In particular $n=n_1+2n_2+\cdots+dn_d$. Although most of the results are quite technical, we believe that they can be used
as an ``experimental data'' for future investigations.

\begin{theorem}
Let $p=1$ and $q=1$.

{\rm (i)} The multiplicity series of $R_{1,1}$ is
\[
M'(R_{1,1};V_d)=\sum_{i=1}^{11}A_i,
\]
where
\[
A_1=\frac{2}{1-v_1};\quad A_2=-\frac{2}{(1-v_1)^2(1-v_2)};
\]
\[
A_3=\frac{v_1v_4v_2-3v_3-v_1-v_4+v_1v_3-3v_2+3v_2v_3+v_2^2-v_2^2v_3+3v_1v_2}{(1-v_1)^3(1-v_2)^3(1-v_3)};
\]
\[
A_4=\frac{v_1+2v_2-v_1v_2+v_3}{(1-v_1)^2(1-v_2)};\quad A_5=\frac{1-v_1v_2}{(1-v_1)^3(1-v_2)^3(1-v_3)};
\]
\[
A_6=\frac{1}{(1-v_1)^2(1-v_2)^2(1-v_3)^2(1-v_4)};\quad A_7=-\frac{v_3+v_4}{1-v_1};
\]
\[
A_8=\frac{(v_1v_2-1)(v_2v_3-1)(v_3v_4-1)}{(1-v_1)^3(1-v_2)^4(1-v_3)^4(1-v_4)^3(1-v_5)};
\]
\[
A_9=\frac{2v_4+v_5+v_3}{(1-v_1)^2(1-v_2)};
\]
\[
A_{10}=\frac{P_{10}}{(1-v_1)^3(1-v_2)^4(1-v_3)^4(1-v_4)^3(1-v_5)},
\]
where
\[
P_{10}=v_1+3v_2-4v_1v_4v_2+4v_3-v_2^2+4v_4-v_1v_3-6v_2v_3+v_2^2v_3-3v_1v_2
\]
\[
-6v_3v_4+3v_5+9v_1v_2v_4v_3-v_3^2+v_6-v_1v_4v_3+v_2^2v_4v_3+6v_2v_4v_3^2-v_2^2v_4v_3^2
\]
\[
+v_1v_4v_3^2+v_1v_2v_3^2-3v_2v_4v_3+3v_1v_2^2v_3-v_1v_2v_3-v_6v_1v_2-v_6v_4v_3
\]
\[
-3v_2v_3v_5+v_2v_3^2v_5-3v_1v_2v_5+v_1v_2^2v_4-v_4^2-v_2v_4-3v_1v_2v_4v_3^2
\]
\[
-3v_1v_2^2v_4v_3^2-v_1v_2^2v_4^2v_3+v_6v_1v_2^2v_3+v_1v_2v_3v_5+3v_1v_2^2v_3v_5-v_1v_2^2v_3^2v_5
\]
\[
+3v_1v_2v_4v_5+3v_2v_4v_3v_5+v_1v_2v_4^2-v_6v_1v_2^2v_4v_3^2-3v_1v_2^2v_4v_3v_5
\]
\[
+v_1v_2^2v_4^2v_3v_5+v_6v_1v_2v_4v_3-v_3v_5+v_2v_4^2v_3-v_6v_2v_3-v_1v_2v_4^2v_5
\]
\[
-v_2v_4^2v_3v_5+v_6v_2v_4v_3^2-3v_4v_5+v_4^2v_5;
\]
\[
A_{11}=\frac{P_{11}}{(1-v_1)^2(1-v_2)},
\]
where
\[
P_{11}=-2v_5-v_1v_5-2v_2v_5+v_1v_2v_5-v_4-2v_1v_4-4v_2v_4+2v_1v_4v_2-v_1v_3
\]
\[
-2v_2v_3+v_1v_2v_3-v_3^2-v_6-v_3v_5-2v_3v_4.
\]

{\rm (ii)} The nonzero multiplicities $m_{\lambda}(R_{1,1})$ are:
\[
m_{\lambda}(R_{1,1})=\frac{3n_1n_2n_3n_4}{4}\left(\frac{n_1n_3n_4+n_1n_2n_4+n_2n_3n_4+n_1n_2n_3}{2}\right.
\]
\[
\left.+\frac{n_2^2n_4+n_2n_3^2+n_1n_3^2+n_2^2n_3}{3}\right)+{\mathcal O}\left(n^6\right),
\]
if $\lambda=(\lambda_1,\lambda_2,\lambda_3,\lambda_4)$;
\[
m_{\lambda}(R_{1,1})={n_1n_2n_3n_4}\left(\frac{n_1n_3n_4+n_1n_2n_4+n_2n_3n_4+n_1n_2n_3}{2}\right.
\]
\[
\left.+\frac{n_2^2n_4+n_2n_3^2+n_1n_3^2+n_2^2n_3}{3}\right)+{\mathcal O}\left(n^6\right),
\]
if $\lambda=(\lambda_1,\lambda_2,\lambda_3,\lambda_4,\lambda_5)$, where $\lambda_5>0$ and
\[
m_{\lambda}(R_{1,1})=\frac{n_1n_2n_3n_4}{4}\left(\frac{n_1n_3n_4+n_1n_2n_4+n_2n_3n_4+n_1n_2n_3}{2}\right.
\]
\[
\left.+\frac{n_2^2n_4+n_2n_3^2+n_1n_3^2+n_2^2n_3}{3}\right)+{\mathcal O}\left(n^6\right),
\]
if $\lambda=(\lambda_1,\lambda_2,\lambda_3,\lambda_4,\lambda_5,1)$.
\end{theorem}

\begin{theorem}
Let $p=0$ and $q=2$.

{\rm (i)} The multiplicity series in two variables of $R_{0,2}$ is
\[
M'(R_{0,2};v_1,v_2)=\frac{P}{(1-v_1)^4(1-v_2)^7},
\]
where
\[
P=1-89v_1^2v_2^3-3v_1+82v_1^2v_2^4-42v_1^2v_2^5-7v_1v_2^6+13v_1^2v_2^6+7v_2v_1^3-21v_2^2v_1^3
\]
\[
+19v_1v_2-52v_1v_2^2+78v_1v_2^3-70v_1v_2^4+30v_1v_2^5-20v_1^2v_2+57v_1^2v_2^2+v_2^7v_1
\]
\[
+35v_2^3v_1^3-35v_2^4v_1^3+21v_2^5v_1^3-7v_2^6v_1^3-2v_2^7v_1^2+v_2^7v_1^3-6v_2+3v_1^2+16v_2^2-v_1^3
\]
\[
-23v_2^3+21v_2^4-8v_2^5+v_2^6.
\]

{\rm (ii)} The multiplicities $m_{\lambda}(R_{0,2})$ for $\lambda=(\lambda_1,\lambda_2)=(n_1+n_2,n_2)$ are
\[
m_{\lambda}=\frac{1}{6!}n_1n_2^4(5n_1^2+6n_1n_2+2n_2^2)+{\mathcal O}\left(n^6\right).
\]
\end{theorem}

We have computed also $M'(R_{0,2};V_3)$, the multiplicities $m_{\lambda}(R_{0,2})$ when $\lambda$ is a partition in three parts
but, besides the expression for $M'(R_{0,2};V_3)$, we give the asymptotics of the multiplicities for $n_3\geq n_1$ only.

\begin{theorem}
Let $p=0$ and $q=2$.

{\rm (i)} The multiplicity series in three variables of $R_{0,2}$ is
\[
M'(R_{0,2};v_1,v_2,v_3)=\sum_{i=1}^{23}A_i,
\]
where
\[
A_1=\frac{2}{(1-v_1)^2(1-v_2)^2(1-v_3)^2};
\]
\[
A_2=-\frac{2v_3}{1-v_1};
\]
\[
A_3=-\frac{3}{(1-v_1)^2(1-v_2)};
\]
\[
A_4=\frac{2}{1-v_1};
\]
\[
A_5=\frac{P_5}{(1-v_3v_1)(1-v_1)^4(1-v_2)^7(1-v_3)^{10}},
\]
where
\[
P_5=141v_1^2v_2^3v_3^3-124v_1^2v_2^2v_3^2-8v_1^3v_2v_3^2-4v_1^3v_2v_3^3-206v_1v_2^2v_3^3+13v_1^3v_2^2v_3^2
\]
\[
+52v_1v_2^2v_3^2+62v_1v_2^2v_3-10v_1v_2^3v_3+38v_1v_2v_3^2+8v_1^2v_2v_3-v_1^3v_2^4v_3^6
\]
\[
+42v_1^3v_2^2v_3^3-2v_1^3v_2^4v_3^5-26v_1v_2v_3-24v_1v_2^4v_3^4+45v_1^3v_2^4v_3^4-8v_1^3v_2^5v_3^4
\]
\[
-4v_1^3v_2^5v_3^5+3v_2v_3^3-2v_1^2v_2^5v_3^5-4v_1v_2^4v_3^5+55v_1v_2^3v_3^4+8v_1^2v_3^3-90v_1^2v_2^3v_3^4+v_1
\]
\[
-40v_1^2v_2v_3^3-28v_1^2v_2v_3^4+4v_2+8v_3-v_1^3v_2^3v_3^6-143v_1v_2^3v_3^2-8v_1^2v_2^2v_3^3
\]
\[
-2v_1^2v_2^2v_3^5-v_1v_2^2v_3^6+27v_1^2v_2^4v_3^5-12v_1v_3^3-2v_3v_1+v_1v_2^3-5v_1v_2+4v_2^4v_3^4
\]
\[
+4v_1^2v_3^4+8v_2^4v_3^3-3v_1^3v_2^2v_3^4+7v_1^2v_2^4v_3+128v_1v_2^3v_3^3-62v_1^3v_2^3v_3^3+10v_1v_2^2v_3^5
\]
\[
-30v_2^3v_3^4v_1^3-8v_2^3v_3^2-v_1v_2^3v_3^6+v_2^2v_3^5-22v_1^2v_2^3v_3+4v_1^2v_2^3v_3^6-5v_1v_2^2-10v_2^2v_3^4
\]
\[
+2v_1v_3^4+104v_1^2v_2^2v_3^4-42v_2v_3^2-45v_2^3v_3^3+11v_2^3v_3^2v_1^3+36v_1^2v_2v_3^2-8v_1v_2^5v_3^2
\]
\[
+90v_1v_2v_3^3+v_2^3v_3^5-v_1^2v_2^5v_3+12v_1^2v_2^5v_3^4-6v_1v_2^3v_3^5+8v_1^3v_2^4v_3^3+2v_2^3v_3^4
\]
\[
-13v_1v_2v_3^4-4v_1v_2^5v_3^3+62v_2^2v_3^2-13v_2v_3+4v_1^2v_2^2-14v_1^2v_2^4v_3^4+4v_3^2-13v_1v_3^2
\]
\[
-4v_1v_2^4v_3^3+2v_1^2v_2^2v_3+92v_1^2v_2^3v_3^2-86v_1^2v_2^4v_3^3-4v_1^3v_2^2v_3+13v_1^2v_2^5v_3^3
\]
\[
+30v_2^2v_3^3+2v_2^3v_3-8v_1v_2^2v_3^4+56v_1v_2^4v_3^2-18v_1^2v_2^4v_3^2+2v_1^2v_2^5v_3^2-29v_1^2v_2^3v_3^5
\]
\[
+10v_1^3v_2^3v_3^5-2v_2^4v_3^2v_1^3-11v_2^2v_3;
\]
\[
A_6=\frac{P_6}{(1-v_1)^3(1-v_2)^4(1-v_3)^4},
\]
where
\[
P_6=2(v_3^3-v_1v_3^3v_2+6v_3^2v_2+v_1v_3^2v_2-3v_1v_3^2v_2^2-4v_3^2-v_3^2v_2^2+v_1v_3^2-3v_2v_3
\]
\[
-v_1v_3+3v_1v_3v_2+v_3v_2^2);
\]
\[
A_7=\frac{P_7}{(1-v_1)^4(1-v_2)^6(1-v_3)^7},
\]
and we have that
\[
P_7=2(-4v_1^2v_2^2-v_1-4v_2-7v_3+v_3^2+v_2^3v_1^2v_3^4+16v_1^2v_2^2v_3^2+v_1v_2v_3^3
\]
\[
+v_1^2v_2v_3^2-26v_2^3v_1^2v_3^2-7v_1^2v_2^2v_3^3+6v_2^3v_1^2v_3+4v_2^4v_1^2v_3^3+v_1^2v_2^2v_3^4+3v_2^4v_1^2v_3^2
\]
\[
+12v_1^2v_2^2v_3+v_2^3v_1^2v_3^3-7v_1^2v_2v_3-v_2^4v_1^2v_3-31v_1v_2v_3^2-43v_1v_2^2v_3
\]
\[
+13v_1v_2v_3+5v_1v_2-v_2^3v_1+5v_1v_2^2-26v_2^2v_3^2-v_2^3v_3+12v_2v_3+23v_1v_2^2v_3^3
\]
\[
+19v_1v_2^2v_3^2-4v_1v_2^2v_3^4-7v_2^4v_1v_3^2-7v_2v_3^3+6v_2^2v_3+16v_2v_3^2+v_2^2v_3^3+v_2^2v_3^4
\]
\[
+3v_2^3v_3^2+4v_2^3v_3^3+v_2^4v_1v_3^3-21v_2^3v_1v_3^3+7v_2^3v_1v_3+27v_2^3v_1v_3^2
\]
\[
+4v_1v_3^2+3v_1v_3+v_2v_3^4);
\]
\[
A_8=\frac{P_8}{(1-v_1)^3(1-v_2)^4(1-v_3)^4},
\]
\[
P_8=2(v_1v_2v_3^2-v_3^2+4v_3-6v_2v_3-v_1v_3+v_2^2v_3+3v_1v_2^2v_3-v_1v_2v_3
\]
\[
+v_1-v_2^2+3v_2-3v_1v_2);
\]
\[
A_9=\frac{v_3^3+v_3^2v_1+2v_3^2v_2-v_3^2v_1v_2}{(1-v_1)^2(1-v_2)};
\]
\[
A_{10}=\frac{2(v_3^2v_2^2-3v_3^2v_2-v_3^2v_1+3v_3^2-v_2^2v_3-3v_3v_1v_2+v_3v_1+3v_2v_3)}{(1-v_1)^3(1-v_2)^3(1-v_3)};
\]
\[
A_{11}=\frac{2v_3}{(1-v_1)}+\frac{2v_3}{(1-v_1)^2}-\frac{2v_3(v_3+2)}{(1-v_1)^2(1-v_2)};
\]
\[
A_{12}=\frac{P_{12}}{(1-v_1)^4(1-v_2)^5(1-v_3)^4},
\]
\[
P_{12}=6v_3-4v_3^2+4v_2+v_1+v_3^3+v_2^3v_1-5v_1v_2-5v_2^2v_1+v_2v_3^3-v_2^2v_3
\]
\[
-11v_2v_3+16v_1v_2v_3^2-4v_1v_2v_3^3-4v_2^3v_1v_3-4v_1v_2v_3+24v_2^2v_1v_3
\]
\[
-11v_2^2v_1^2v_3-v_2^3v_1^2v_3-4v_1^2v_2v_3^2+6v_1^2v_2v_3+4v_2^3v_1^2v_3^2+v_2^2v_1^2v_3^3-16v_1v_2^2v_3^2
\]
\[
+v_1^2v_2v_3^3+4v_2^2v_3^2+4v_2^2v_1^2-4v_1v_3;
\]
\[
A_{13}=\frac{2(3v_2v_3-v_2^2v_3-3v_3+v_1v_3-v_1+3v_1v_2+v_2^2-3v_2)}{(1-v_1)^3(1-v_2)^3(1-v_3)};
\]
\[
A_{14}=\frac{v_3+v_1+2v_2-v_1v_2}{(1-v_1)^2(1-v_2)};
\]
\[
A_{15}=\frac{P_{15}}{(1-v_1v_3)(1-v_1)^4(1-v_2)^7(1-v_3)^{10}},
\]
\[
P_{15}=-1+34v_1^2v_2^2v_3^2+21v_1v_2^3v_3+30v_1^2v_2^4v_3^3-30v_1v_2v_3^2+34v_1v_2^2v_3^2
\]
\[
-37v_1v_2^2v_3+v_1v_2^4v_3^5-6v_1^2v_2v_3-7v_1^2v_2^4v_3^4-6v_1^2v_2^5v_3^3-10v_1v_2^4v_3^2-2v_2^4v_3^3
\]
\[
-v_1^2v_2^5v_3^4+2v_1v_2^5v_3^2+2v_1v_2^4v_3^3-7v_1^3v_2^2v_3^3+v_1^3v_2v_3^3+2v_1^3v_2^5v_3^4
\]
\[
+v_1^3v_2^4v_3^5+v_1^3v_2^5v_3^5-7v_1^3v_2^4v_3^3-10v_1^3v_2^4v_3^4-10v_1^3v_2^2v_3^2+2v_1^3v_2v_3^2+v_1^3v_2^2v_3
\]
\[
+v_1^3v_2^5v_3^3-v_1^2v_2^2+4v_1v_2-v_1^2v_2+v_1^3v_2v_3+7v_1v_2v_3+37v_1^2v_2^3v_3^4+24v_1^3v_2^3v_3^3
\]
\[
-2v_1^2v_3^3-24v_2^2v_3^2+v_1v_3-v_1^2v_3^2+6v_1v_3^2-v_2-2v_3-v_3^2+v_1v_2^5v_3^3-35v_1^2v_2^3v_3^2
\]
\[
+9v_1^2v_2^4v_3^2-v_1^2v_2^5v_3^2-7v_1v_2^4v_3+v_1v_2^5v_3+10v_1^2v_2^2v_3+35v_1v_2^2v_3^3-34v_1^2v_2^3v_3^3
\]
\[
-14v_1^2v_2^2v_3^3+14v_1v_2^3v_3^2-21v_1^2v_2^2v_3^4-34v_1v_2^3v_3^3+v_1v_2^3v_3^5-10v_1v_2^3v_3^4
\]
\[
-9v_1v_2v_3^3-2v_1^2v_2v_3^2-v_1^2v_3^4+v_1v_3^3+10v_2^3v_3^3-v_2^4v_3^2+7v_2^3v_3^2+7v_2v_3^2
\]
\[
+10v_2v_3+10v_1^2v_2v_3^3+7v_1^2v_2v_3^4-4v_1^2v_2^4v_3^5+6v_1v_2^4v_3^4-v_2^3v_3^4-v_2^4v_3^4;
\]
\[
A_{16}=\frac{2v_3(v_2v_1-1)(v_2v_3-1)}{(1-v_1)^3(1-v_2)^4(1-v_3)^4};
\]
\[
A_{17}=\frac{P_{17}}{(1-v_1)^4(1-v_2)^6(1-v_3)^7},
\]
\[
P_{17}=2(1+v_1^2v_2v_3-v_1^2v_2^4v_3^2-7v_1^2v_2^2v_3+v_3+v_2+7v_1^2v_2^3v_3^2-v_1^2v_2^3v_3^3-v_1v_3
\]
\[
+7v_2^2v_3^2-7v_2v_3-v_2^3v_3^2-v_2^2v_3^3-v_2^3v_3^3-v_1v_3^2+v_1^2v_2^2+v_1^2v_2-2v_1v_2^3v_3^2+v_1v_2^4v_3^2
\]
\[
-6v_1v_2^3v_3-v_1^2v_2^4v_3^3+16v_1v_2^2v_3-16v_1v_2^2v_3^2+v_1v_2^4v_3+2v_1v_2v_3+4v_1v_2^3v_3^3
\]
\[
+6v_1v_2v_3^2-4v_1v_2);
\]
\[
A_{18}=-\frac{2(v_1v_2-1)(v_2v_3-1)}{(1-v_1)^3(1-v_2)^4(1-v_3)^4};
\]
\[
A_{19}=-\frac{v_3^2}{(1-v_1)^2(1-v_2)};
\]
\[
A_{20}=\frac{2v_3(v_2v_1-1)}{(1-v_1)^3(1-v_2)^3(1-v_3)};
\]
\[
A_{21}=\frac{2v_3}{(1-v_1)^2(1-v_2)};
\]
\[
A_{22}=\frac{P_{22}}{(1-v_1)^4(1-v_2)^5(1-v_3)^4},
\]
\[
P_{22}=-v_1^2v_2^3v_3^2-v_1^2v_2^2v_3^2+4v_1v_2^2v_3^2+v_1v_2^3v_3+4v_3v_2^2v_1^2-v_2^2v_3^2
\]
\[
-5v_1v_2^2v_3-v_1^2v_2^2-v_2v_3^2-5v_1v_2v_3-v_2v_1^2+4v_2v_3+4v_1v_2+v_1v_3-v_2-1;
\]
\[
A_{23}=\frac{2(1-v_1v_2)}{(1-v_1)^3(1-v_2)^3(1-v_3)}.
\]

{\rm (ii)} The multiplicities $m_{\lambda}(R_{0,2})$ when $\lambda$ is a partition in three parts and $n_3\geq n_1$ are
\[
m_{\lambda}(R_{0,2})=4\frac{n_1^2n_2^5n_3^5}{2!5!5!}+16\frac{n_1^2n_2^2n_3^8}{2!2!8!}+6\frac{n_1^2n_2n_3^9}{2!1!9!}
+18\frac{n_1^2n_2^3n_3^7}{2!3!7!}+12\frac{n_1^2n_2^4n_3^6}{2!4!6!}
\]
\[
+2\frac{n_2(n_3-n_1)^{11}}{1!11!}-2\frac{n_2n_3^{11}}{1!11!}+10\frac{n_1^3n_2n_3^8}{3!1!8!}
+10\frac{n_1^3n_2^2n_3^7}{3!2!7!}+6\frac{n_1^3n_2^3n_3^6}{3!3!6!}+2\frac{n_1^3n_2^4n_3^5}{3!4!5!}
\]
\[
+8\frac{n_1n_2^2n_3^9}{1!2!9!}+18\frac{n_1n_2^4n_3^7}{1!4!7!}+4\frac{n_1n_2^6n_3^5}{1!6!5!}
+16\frac{n_1n_2^3n_3^8}{1!3!8!}+2\frac{n_1n_2n_3^{10}}{1!1!10!}+12\frac{n_1n_2^5n_3^6}{1!5!6!}
\]
\[
+{\mathcal O}\left(n^{11}\right).
\]
\end{theorem}

\begin{theorem}
Let $p=0$ and $q=3$.

{\rm (i)} The multiplicity series in two variables of $R_{0,3}$ is
\[
M'(R_{0,3};v_1,v_2)=\sum_{i=1}^{38}A_i,
\]
where
\[
A_1=\frac{3(4v_1^2v_2^2-5v_1v_2-5v_1v_2^2+v_1v_2^3+v_1+4v_2)}{(1-v_1)^4(1-v_2)^7};
\]
\[
A_2=-\frac{12(4v_1^2v_2^2-5v_1v_2-5v_1v_2^2+v_1v_2^3+v_1+4v_2)}{(1-v_1)^4(1-v_2)^6};
\]
\[
A_3=\frac{6(v_1-3v_1v_2+3v_2-v_2^2)}{(1-v_1)^3(1-v_2)^4};
\]
\[
A_4=\frac{P_4}{(1-v_1)^5(1-v_2)^9},
\]
\[
P_4=6(5v_1^3v_2^3+5v_2^2v_1^3+v_1^2v_2^4-24v_1^2v_2^2-7v_1^2v_2^3+7v_1v_2
\]
\[
+24v_1v_2^2-v_1-5v_2^2-5v_2);
\]
\[
A_5=\frac{P_5}{(1-v_1)^5(1-v_2)^8},
\]
where
\[
P_5=-12(5v_1^3v_2^3+5v_2^2v_1^3+v_1^2v_2^4-24v_1^2v_2^2-7v_1^2v_2^3+7v_1v_2
\]
\[
+24v_1v_2^2-v_1-5v_2^2-5v_2);
\]
\[
A_6=\frac{P_6}{(1-v_1)^6(1-v_2)^{10}},
\]
\[
P_6=-6(16v_2^3v_1^4+6v_2^4v_1^4+6v_2^2v_1^4-9v_1^3v_2^4+v_1^3v_2^5-35v_2^2v_1^3-69v_1^3v_2^3
\]
\[
+84v_1^2v_2^3+84v_1^2v_2^2-9v_1v_2-35v_1v_2^3-69v_1v_2^2+v_1+16v_2^2+6v_2+6v_2^3);
\]
\[
A_7=\frac{P_7}{(1-v_1)^6(1-v_2)^{12}},
\]
\[
P_7=27v_2^4v_1^4+v_2^5v_1^4+v_2^2v_1^4+27v_2^3v_1^4-74v_1^3v_2^4-150v_1^3v_2^3-6v_1^3v_2^2
\]
\[
+6v_1^3v_2^5+v_1^2+51v_1^2v_2^2+250v_1^2v_2^3+51v_1^2v_2^4-9v_1^2v_2^5+v_1^2v_2^6-9v_1^2v_2
\]
\[
+6v_1v_2-150v_1v_2^3-74v_1v_2^2-6v_1v_2^4+v_2+27v_2^3+v_2^4+27v_2^2;
\]
\[
A_8=\frac{P_8}{(1-v_1)^6(1-v_2)^{11}},
\]
\[
P_8=-3(27v_2^4v_1^4+v_2^5v_1^4+v_2^2v_1^4+27v_2^3v_1^4-74v_1^3v_2^4-150v_1^3v_2^3-6v_1^3v_2^2
\]
\[
+6v_1^3v_2^5+v_1^2+51v_1^2v_2^2+250v_1^2v_2^3+51v_1^2v_2^4-9v_1^2v_2^5+v_1^2v_2^6-9v_1^2v_2
\]
\[
+6v_1v_2-150v_1v_2^3-74v_1v_2^2-6v_1v_2^4+v_2+27v_2^3+v_2^4+27v_2^2);
\]
\[
A_9=\frac{P_9}{(1-v_1)^5(1-v_2)^9},
\]
\[
P_9=-3(v_1^3v_2^4+v_1^3v_2^2+18v_1^3v_2^3-45v_1^2v_2^3-26v_1^2v_2^2+5v_1^2v_2^4-v_1^2+7v_1^2v_2
\]
\[
+v_1v_2^5+45v_1v_2^2-7v_1v_2^4-5v_1v_2+26v_1v_2^3-v_2^3-18v_2^2-v_2);
\]
\[
A_{10}=\frac{P_{10}}{(1-v_1)^6(1-v_2)^{10}},
\]
\[
P_{10}=3(27v_2^4v_1^4+v_2^5v_1^4+v_2^2v_1^4+27v_2^3v_1^4-74v_1^3v_2^4-150v_1^3v_2^3-6v_1^3v_2^2
\]
\[
+6v_1^3v_2^5+v_1^2+51v_1^2v_2^2+250v_1^2v_2^3+51v_1^2v_2^4-9v_1^2v_2^5+v_1^2v_2^6-9v_1^2v_2
\]
\[
+6v_1v_2-150v_1v_2^3-74v_1v_2^2-6v_1v_2^4+v_2+27v_2^3+v_2^4+27v_2^2);
\]
\[
A_{11}=\frac{P_{11}}{(1-v_1)^5(1-v_2)^8},
\]
\[
P_{11}=6(v_1^3v_2^4+v_1^3v_2^2+18v_1^3v_2^3-45v_1^2v_2^3-26v_1^2v_2^2+5v_1^2v_2^4-v_1^2+7v_1^2v_2
\]
\[
+v_1v_2^5+45v_1v_2^2-7v_1v_2^4-5v_1v_2+26v_1v_2^3-v_2^3-18v_2^2-v_2);
\]
\[
A_{12}=\frac{P_{12}}{(1-v_1)^4(1-v_2)^6},
\]
\[
P_{12}=3(v_1^2v_2^3+11v_1^2v_2^2-5v_1^2v_2+v_1^2-24v_1v_2^2+4v_1v_2^3+4v_1v_2
\]
\[
+v_2+v_2^4+11v_2^2-5v_2^3);
\]
\[
A_{13}=\frac{P_{13}}{(1-v_1)^6(1-v_2)^{12}},
\]
\[
P_{13}=-2(16v_2^3v_1^4+6v_2^4v_1^4+6v_2^2v_1^4-9v_1^3v_2^4+v_1^3v_2^5-35v_2^2v_1^3-69v_1^3v_2^3
\]
\[
+84v_1^2v_2^3+84v_1^2v_2^2-9v_1v_2-35v_1v_2^3-69v_1v_2^2+v_1+16v_2^2+6v_2+6v_2^3);
\]
\[
A_{14}=\frac{P_{14}}{(1-v_1)^6(1-v_2)^{11}},
\]
\[
P_{11}=6(16v_2^3v_1^4+6v_2^4v_1^4+6v_2^2v_1^4-9v_1^3v_2^4+v_1^3v_2^5-35v_2^2v_1^3-69v_1^3v_2^3
\]
\[
+84v_1^2v_2^3+84v_1^2v_2^2-9v_1v_2-35v_1v_2^3-69v_1v_2^2+v_1+16v_2^2+6v_2+6v_2^3);
\]
\[
A_{15}=\frac{8(-v_1+3v_1v_2-3v_2+v_2^2)}{(1-v_1)^3(1-v_2)^3};
\]
\[
A_{16}=\frac{9(v_1^2v_2^2+v_2v_1^2-4v_1v_2+v_2+1)}{(1-v_1)^4(1-v_2)^6};
\]
\[
A_{17}=\frac{3}{(1-v_1)^2(1-v_2)^2};
\]
\[
A_{18}=\frac{3(v_1-v_1v_2+2v_2)}{(1-v_1)^2(1-v_2)};
\]
\[
A_{19}=\frac{6(v_1v_2-1)}{(1-v_1)^3(1-v_2)^4};
\]
\[
A_{20}=\frac{P_{20}}{(1-v_1)^6(1-v_2)^{10}},
\]
\[
P_{20}=3(v_2^4v_1^4+6v_2^2v_1^4+v_2v_1^4+6v_2^3v_1^4-20v_2^3v_1^3-30v_2^2v_1^3-6v_2v_1^3
\]
\[
+54v_1^2v_2^2+15v_1^2v_2^3+15v_2v_1^2-6v_1v_2^3-30v_1v_2^2-20v_1v_2+v_2^3+1+6v_2^2+6v_2);
\]
\[
A_{21}=-\frac{3(v_1^2v_2^2+v_2v_1^2-4v_1v_2+v_2+1)}{(1-v_1)^4(1-v_2)^7};
\]
\[
A_{22}=\frac{9(4v_1^2v_2^2-5v_1v_2-5v_1v_2^2+v_1v_2^3+v_1+4v_2)}{(1-v_1)^4(1-v_2)^5};
\]
\[
A_{23}=\frac{P_{23}}{(1-v_1)^3(1-v_2)^3},
\]
\[
P_{23}=3v_1^2v_2^2-3v_1^2v_2-v_1^2v_2^3+v_1^2-10v_1v_2^2+3v_1v_2+3v_1v_2^3+6v_2^2-3v_2^3+v_2;
\]
\[
A_{24}=\frac{P_{24}}{(1-v_1)^5(1-v_2)^9},
\]
\[
P_{24}=3(1-v_2v_1^3-v_2^3v_1^3-3v_2^2v_1^3+10v_1^2v_2^2+5v_2v_1^2-5v_1v_2^2-10v_1v_2+v_2^2+3v_2);
\]
\[
A_{25}=\frac{P_{25}}{(1-v_1)^6(1-v_2)^{11}},
\]
\[
P_{25}=-3(v_2^4v_1^4+6v_2^2v_1^4+v_2v_1^4+6v_2^3v_1^4-20v_2^3v_1^3-30v_2^2v_1^3-6v_2v_1^3
\]
\[
+54v_1^2v_2^2+15v_1^2v_2^3+15v_2v_1^2-6v_1v_2^3-30v_1v_2^2-20v_1v_2+v_2^3+1+6v_2^2+6v_2);
\]
\[
A_{26}=\frac{P_{26}}{(1-v_1)^6(1-v_2)^{9}},
\]
\[
P_{26}=2(6v_2^2v_1^4+6v_2^4v_1^4+16v_2^3v_1^4+v_1^3v_2^5-35v_2^2v_1^3-69v_1^3v_2^3-9v_1^3v_2^4
\]
\[
+84v_1^2v_2^3+84v_1^2v_2^2-9v_1v_2-69v_1v_2^2-35v_1v_2^3+v_1+16v_2^2+6v_2^3+6v_2);
\]
\[
A_{27}=\frac{P_{27}}{(1-v_1)^5(1-v_2)^8},
\]
\[
P_{27}=6(v_2v_1^3+3v_2^2v_1^3+v_2^3v_1^3-5v_2v_1^2-10v_1^2v_2^2+5v_1v_2^2+10v_1v_2-1-v_2^2-3v_2);
\]
\[
A_{28}=\frac{P_{28}}{(1-v_1)^6(1-v_2)^{12}},
\]
\[
P_{28}=v_2^4v_1^4+6v_2^2v_1^4+v_2v_1^4+6v_2^3v_1^4-20v_2^3v_1^3-30v_2^2v_1^3-6v_2v_1^3
\]
\[
+54v_1^2v_2^2+15v_1^2v_2^3+15v_2v_1^2-6v_1v_2^3-30v_1v_2^2-20v_1v_2+v_2^3+1+6v_2^2+6v_2;
\]
\[
A_{29}=\frac{P_{29}}{(1-v_1)^5(1-v_2)^7},
\]
\[
P_{29}=6(5v_1^3v_2^3+5v_2^2v_1^3-7v_1^2v_2^3+v_1^2v_2^4-24v_1^2v_2^2-v_1+7v_1v_2
\]
\[
+24v_1v_2^2-5v_2-5v_2^2);
\]
\[
A_{30}=\frac{P_{30}}{(1-v_1)^6(1-v_2)^9},
\]
\[
P_{30}=-(27v_2^4v_1^4+v_2^2v_1^4+27v_2^3v_1^4+v_2^5v_1^4-150v_1^3v_2^3-6v_1^3v_2^2-74v_1^3v_2^4
\]
\[
+6v_1^3v_2^5+v_1^2v_2^6-9v_1^2v_2+250v_1^2v_2^3+51v_1^2v_2^2-9v_1^2v_2^5+v_1^2+51v_1^2v_2^4-74v_1v_2^2
\]
\[
-150v_1v_2^3+6v_1v_2-6v_1v_2^4+27v_2^3+v_2^4+v_2+27v_2^2);
\]
\[
A_{31}=\frac{P_{31}}{(1-v_1)^5(1-v_2)^7},
\]
\[
P_{31}=-3(18v_1^3v_2^3+v_1^3v_2^2+v_1^3v_2^4-45v_1^2v_2^3-26v_1^2v_2^2+7v_1^2v_2-v_1^2+5v_1^2v_2^4
\]
\[
-7v_1v_2^4+26v_1v_2^3+v_1v_2^5-5v_1v_2+45v_1v_2^2-18v_2^2-v_2^3-v_2);
\]
\[
A_{32}=\frac{P_{32}}{(1-v_1)^4(1-v_2)^5},
\]
\[
P_{32}=-3(v_1^2v_2^3+11v_1^2v_2^2-5v_1^2v_2+v_1^2-24v_1v_2^2+4v_1v_2^3+4v_1v_2
\]
\[
+v_2+v_2^4+11v_2^2-5v_2^3);
\]
\[
A_{33}=\frac{3}{1-v_1};
\]
\[
A_{34}=-\frac{6}{(1-v_1)^2(1-v_2)};
\]
\[
A_{35}=-\frac{6(v_1^2v_2^2+v_2v_1^2-4v_1v_2+v_2+1)}{(1-v_1)^4(1-v_2)^5};
\]
\[
A_{36}=\frac{7(1-v_1v_2)}{(1-v_1)^3(1-v_2)^3};
\]
\[
A_{37}=\frac{P_{37}}{(1-v_1)^5(1-v_2)^7},
\]
\[
P_{37}=-3(v_2v_1^3+3v_2^2v_1^3+v_2^3v_1^3-5v_2v_1^2-10v_1^2v_2^2+5v_1v_2^2+10v_1v_2-1-v_2^2-3v_2);
\]
\[
A_{38}=\frac{P_{38}}{(1-v_1)^6(1-v_2)^9},
\]
\[
P_{38}=-(v_2^4v_1^4+6v_2^2v_1^4+v_2v_1^4+6v_2^3v_1^4-20v_2^3v_1^3-30v_2^2v_1^3-6v_2v_1^3
\]
\[
+54v_1^2v_2^2+15v_1^2v_2^3+15v_2v_1^2-6v_1v_2^3-30v_1v_2^2-20v_1v_2+v_2^3+1+6v_2^2+6v_2).
\]

{\rm (ii)} For $\lambda$ being a partition in two parts
the multiplicities of $R_{3,0}$ are
\[
m_{\lambda}(R_{3,0})=\frac{1}{11!}n_1n_2^7(66n_1^4+77n_1n_2^3+165n_1^2n_2^2+165n_1^3n_2+14n_2^4)+{\mathcal O}\left(n^{11}\right).
\]
\end{theorem}

\begin{theorem}
Let $p=0$, $q=4$ and let $\lambda$ be a partition in two parts. Then the multiplicities $m_{\lambda}(R_{0,4})$ are
\[
m_{\lambda}(R_{0,4})=\frac{1}{10!}n_1n_2^6(n_1+n_2)^6+{\mathcal O}\left(n^{12}\right).
\]
\end{theorem}

\end{document}